\documentclass{amsproc}

\usepackage{amsmath,amssymb,amsthm,amsfonts}
\usepackage{amscd}
\usepackage{tikz}
\usepackage{float}
\usepackage{keyval}
\usepackage{caption}
\usepackage{subcaption}
\usepackage[english]{babel}

\newtheorem{theorem}{Theorem}[section]
\newtheorem{proposition}[theorem]{Proposition}
\newtheorem{lemma}[theorem]{Lemma}
\newtheorem*{notation}{Notation}

\theoremstyle{definition}
\newtheorem{definition}[theorem]{Definition}

\newtheorem{question}{Question}

\theoremstyle{remark}

\numberwithin{equation}{section}

\newcommand{\abs}[1]{\lvert#1\rvert}

\newcommand\Rp{\mathbb{R}^+}
\newcommand\N{\mathbb N}
\newcommand\No{{\mathbb{N}_0}}
\newcommand\C{\mathbb C}
\newcommand\Z{\mathbb Z}

\newcommand\pe[1]{\varepsilon_{#1}}

\newcommand\dist{\mathrm{dist}}
\newcommand\ob[2]{\mathrm{B}(#1,#2)}
\newcommand\cb[2]{\bar{\mathrm{B}}(#1,#2)}
\newcommand\scs{\mathcal F}
\newcommand\ar{\rho_{\rm ann}}

\DeclareMathOperator\inte{int} 

\begin{document}

\title{Swiss Cheeses and Their Applications}

\author{J. F. Feinstein}
\address{School of Mathematical Sciences, University of Nottingham, University Park, Nottingham, NG7 2RD, UK}
\email{joel.feinstein@nottingham.ac.uk}

\author{S. Morley}
\address{School of Mathematical Sciences, University of Nottingham, University Park, Nottingham, NG7 2RD, UK}
\email{pmxsm9@nottingham.ac.uk}
\thanks{The second author was supported by a grant from the EPSRC}

\author{H. Yang}
\address{School of Mathematical Sciences, University of Nottingham, University Park, Nottingham, NG7 2RD, UK}
\email{pmxhy1@nottingham.ac.uk}
\thanks{The third author was supported by a China Tuition Fee Research Scholarship and the School of Mathematical Sciences at the University of Nottingham.}

\subjclass[2000]{Primary 46J10; Secondary 54H99}
\date{\today}


\keywords{Swiss cheeses, rational approximation, uniform algebras, derivations, weak amenability}

\begin{abstract}
Swiss cheese sets have been used in the literature as useful examples in the study of rational approximation and uniform algebras. In this paper, we give a survey of Swiss cheese constructions and related results. We describe some notable examples of Swiss cheese sets in the literature. We explain the various abstract notions of Swiss cheeses, and how they can be manipulated to obtain desirable properties. In particular, we discuss the Feinstein-Heath classicalisation theorem and related results. We conclude with the construction of a new counterexample to a conjecture of S. E. Morris, using a classical Swiss cheese set.
\end{abstract}

\maketitle
\section{Introduction}
In this introduction, we briefly discuss some notions which will be explained in more detail later, and we describe the structure of the paper. It is a classical problem of approximation theory to determine when continuous functions on a compact subset of the complex plane can be uniformly approximated by rational functions. As an extension of this problem, we often consider the algebra of those continuous functions which can be approximated in this way. Examples of compact subsets of the complex plane on which this algebra has interesting properties are obtained by deleting a sequence of carefully chosen open disks from a closed disk. We call such sets {\em Swiss cheese sets}. We are usually interested in Swiss cheese sets where the sum of the radii of the deleted open disks is finite. In this case, if the deleted open disks and the complement of the closed disk have pairwise disjoint closures, then we call the set a {\em classical} Swiss cheese set. In the literature, a Swiss cheese traditionally refers to a compact subset of the plane; whereas, in \cite{feinheath2010,FMY,mason2010} and the current paper, we consider a Swiss cheese to be an underlying object to which we associate a Swiss cheese set.

In general, a Swiss cheese set might have some undesirable topological properties, such as the existence of isolated points. However, Feinstein and Heath \cite{feinheath2010} described a process by which, under certain conditions, we can improve the topological properties of the set. We call this process ``classicalisation''. We discuss the desirable properties of classical Swiss cheese sets in Section \ref{Swiss_cheese_sec}.

In \cite{feinheath2010}, Feinstein and Heath considered a {\em Swiss cheese} to be a pair consisting of a closed disk and a countable collection of open disks. A Swiss cheese set is associated to such an object by deleting the union of the collection of open disks from the closed disk. They then used Zorn's lemma to prove the Feinstein-Heath classicalisation theorem. In \cite{mason2010}, Mason used transfinite induction to prove the classicalisation theorem. His proof used maps, which he called ``disk assignment functions'', which can also be used to describe Swiss cheese sets. In \cite{FMY}, the current authors considered sequences of pairs, consisting of a complex number and a non-negative real number, each of which corresponds to the centre and radius of a disk in the plane. Using these sequences, which we call ``abstract Swiss cheeses'', we gave a topological proof of the Feinstein-Heath classicalisation theorem, and some related results.

In this paper we describe numerous examples of Swiss cheese sets from the literature, and outline some classicalisation results from \cite{feinheath2010,FMY,mason2010}. We also give a short comparison of the various abstract objects used to describe and manipulate Swiss cheese sets and the methods used to prove the Feinstein-Heath classicalisation theorem. We then give a proof of a ``semiclassicalisation'' theorem, where the open disks and the complement of the closed disk of the final Swiss cheese set are pairwise disjoint. This proof uses an inductive construction which terminates at the first infinite ordinal. We conclude with the construction of a classical Swiss cheese set which serves as a counterexample to a conjecture of S. E. Morris, and improves the example in \cite{feinstein2004}.

\section{Preliminaries}
We say {\em compact space} to mean a non-empty, compact, Hausdorff topological space and {\em compact plane set} to mean a non-empty, compact subset of $\C$. For a commutative, unital Banach algebra $A$ we denote the character space of $A$ by $\Phi_A$. Given the Gelfand topology, $\Phi_A$ is a compact space. Let $X$ be a compact space. We denote the algebra of continuous, complex-valued functions on $X$ by $C(X)$ and we give $C(X)$ the {\em uniform norm} on $X$ defined by
\[
\abs f_X:=\sup_{x\in X}\abs{f(x)}\qquad (f\in C(X)).
\]
A {\em uniform algebra on $X$} is a closed subalgebra $A$ of $C(X)$ which contains all constant functions and for each $x,y\in X$ with $x\neq y$ there is a function $f\in A$ such that $f(x)\neq f(y)$. We say a uniform algebra $A$ is {\em trivial} if $A=C(X)$. We say $A$ is {\em natural} if every character on $A$ corresponds to a {\em point evaluation} $\pe x(f):=f(x)$ ($f\in A$) for some $x\in X$. In general, we shall often identify the point $x$ with the point evaluation $\pe x$. For general background on uniform algebras we refer the reader to \cite{browder1969,gamelin1984,stout1971} and for general background on Banach algebras we refer the reader to \cite{dales2000}.

Let $X$ be a compact plane set. We denote the set of all rational functions with poles off $X$ by $R_0(X)$ and we denote its closure in $C(X)$ by $R(X)$. We denote the set of all functions $f\in C(X)$ with $f|_{\inte X}$ analytic by $A(X)$. It is easy to see that $R(X)$ and $A(X)$ are uniform algebras on $X$,
\[R_0(X)\subseteq R(X)\subseteq A(X)\subseteq C(X),\]
and $A(X)=C(X)$ if and only if $\inte X$ is empty. It is standard (see, for example, \cite[Chapter II]{gamelin1984}) that $R(X)$, $A(X)$ and $C(X)$ are natural uniform algebras on $X$.

\begin{definition}
Let $A$ be a commutative$,$ unital Banach algebra. Then $A$ is {\em regular} if for each closed $E\subseteq\Phi_A$ and $\varphi\in\Phi_A\setminus E$ there exists $\hat f\in\hat A,$ the Gelfand transform of $A,$ such that $\hat f(\varphi)=1$ and $\hat f(E)\subseteq\{0\}$. Similarly$,$ $A$ is {\em normal} if for each pair of disjoint$,$ closed sets $E,F\subseteq\Phi_A$ there exists $\hat f\in\hat A$ such that $\hat f(E)\subseteq\{1\}$ and $\hat f(F)\subseteq\{0\}$.
\end{definition}

It is standard that regularity and normality are equivalent for commutative, unital Banach algebras (see, for example, \cite[Proposition~4.1.18]{dales2000}).

\begin{definition}
Let $A$ be a commutative Banach algebra and $E$ a commutative Banach $A$-bimodule. A {\em derivation} $D:A\to E$ is a linear map such that
\[
D(ab)=a\cdot D(b)+D(a)\cdot b\qquad(a,b\in A).
\]
A commutative Banach algebra is {\em weakly amenable} if there are no non-zero continuous derivations from $A$ into any commutative Banach $A$-bimodule.
\end{definition}

We refer the reader to \cite{dales2000} for the definition of Banach $A$-bimodules and further details. It is known (\cite{bade1987}) that a commutative Banach algebra $A$ is weakly amenable if and only if there are no non-zero, continuous derivations from $A$ into $A'$, the dual module of $A$.

\begin{definition}
Let $A$ be a commutative, unital Banach algebra and let $\varphi$ be a character on $A$. A {\em point derivation} at $\varphi$ is a linear functional $d$ on $A$ such that
\[
d(ab)=\varphi(a)d(b)+d(a)\varphi(b)\qquad(a,b\in A).
\]
A {\em point derivation} of order $n\in \mathbb N$ (respectively, $\infty$) at $\varphi$ is a sequence $d_0,d_1,\cdots$ of linear functionals with $d_0=\varphi,$ satisfying
\[
d_j(ab) = \sum\limits_{k=0}^jd_k(a)d_{j-k}(b)\qquad(a,b\in A),
\]
for $j=1,2,\dotsb,n$ (respectively, $j=1,2,\cdots$).
A point derivation is continuous if $d_j$ is continuous for $j\leq n$ (respectively, all $n$).
\end{definition}

Let $A$ be an algebra of functions on a compact space $X$. A point $x\in X$ is a {\em peak point} for $A$ if there exists $f\in A$ such that $f(x)=1$ and $\abs{f(y)}<1$ for all $y\in X$ with $y\neq x$. It is standard that if $x$ is a peak point for $A$ then there are no non-zero point derivations at $x$, see \cite[Corollary 1.6.7]{browder1969}.

\begin{definition}
A uniform algebra $A$ on a compact space $X$ is {\em essential} if, for each non-empty, proper, closed subset $F\subseteq X,$ there is a function $f\in C(X)\setminus A$ such that $f|_F=0$.
\end{definition}

It is standard (\cite[Theorem~2.8.1]{browder1969}) that $A$ is essential if and only if the union of all of the supports for annihilating measures for $A$ on $X$ is dense in $X$.

\begin{definition}
Let $A$ be a uniform algebra on a compact space $X$. Then $A$ is {\em antisymmetric} if every real valued function in $A$ is constant.
\end{definition}

We remark that every antisymmetric uniform algebra is essential, but the converse is false (\cite[Page~147]{browder1969}).

A number of the constructions in this paper involve finding a compact subset of a given compact plane set so that the subset has better topological properties. The following proposition, from \cite{feinheath2010}, lists some properties of $R(X)$ which are preserved when we consider a subset $Y$ of $X$.

\begin{proposition}\label{subsetprops}
Let $X$ and $Y$ be compact plane sets with $Y\subseteq X$. Then$:$
\begin{enumerate}
  \item if $R(X)$ is trivial then so is $R(Y);$
  \item if $R(X)$ does not have any non-zero bounded point derivations then neither does $R(Y);$
  \item if $R(X)$ is regular then so is $R(Y)$.
\end{enumerate}
\end{proposition}

\section{Swiss cheeses}
\label{Swiss_cheese_sec}
Some of the properties that a uniform algebra can possess depend on the nature of its character space. Swiss cheese sets (as described in the introduction) are often constructed so that $R(X)$ has various combinations of desirable properties. We discuss various such constructions in Section \ref{swisscheeseexamples}.

In \cite{feinheath2010,mason2010} a {\em Swiss cheese} object was associated to a Swiss cheese set to allow the manipulation of the closed disk and each open disk. These objects are a pair consisting of a closed disk and a countable collection of open disks, where the associated Swiss cheese set is obtained by deleting each open disk in the collection from the closed disk. These objects can be associated to compact plane sets which previously may not have been considered Swiss cheese sets. In fact, without some additional conditions, every compact plane set is a Swiss cheese set. We usually insist the sum of the radii of the open disks is finite to limit the sets we can obtain from Swiss cheeses.

Another way of studying Swiss cheese sets was later developed (in \cite{FMY}) in which a sequence of pairs, consisting of a complex number and a non-negative real number, are associated to a Swiss cheese set. The elements in each pair of these sequences correspond to the centres and radii, respectively, of the disks used to construct a Swiss cheese set. In this way the space of {\em abstract Swiss cheeses} can be given a topology which can be used to prove various results about the associated Swiss cheese sets.

\begin{notation}For $a\in \C$ and $r> 0$ we denote the open disk centred at $a$ of radius $r$ by $\ob{a}{r}$ and the corresponding closed disk by $\cb{a}{r}$. We also set $\ob a0=\emptyset$ and $\cb a0=\{a\}$. We denote by $\No$ the set of non-negative integers, $\N$ the set of positive integers and $\Rp$ the set of non-negative real numbers.
\end{notation}

As in \cite{FMY}, we define the space $\scs:=(\C\times\Rp)^\No$ with the product topology.

\begin{definition}\label{absSCdef}
Let $A=((a_n,r_n))_{n=0}^\infty\in\scs$. We call $A$ an {\em abstract Swiss cheese }$,$ and we define the following.
\begin{enumerate}
 \item The \emph{significant index set} of $A$ is  $S_A:=\{n\in \mathbb N:r_n>0\}$. We say that $A$ is {\em finite} if $S_A$ is a finite set, otherwise it is \emph{infinite}.
 \item The {\em associated Swiss cheese set} $X_A$ is given by
 \begin{equation}\label{absScs}X_A:=\cb{a_0}{r_0}\setminus\left( \bigcup\limits_{n=1}^\infty\ob{a_n}{r_n}\right).
 \end{equation}
 \item We say that $A$ is {\em semiclassical} if $\sum_{n=1}^\infty r_n<\infty,$ $r_0>0$ and for all $k\in S_A$  the following holds$:$
     \begin{enumerate}
        \item $\ob{a_k}{r_k}\subseteq\ob{a_0}{r_0};$
        \item whenever $\ell\in S_A$ has $\ell\neq k,$ we have $\ob{a_k}{r_k}\cap\ob{a_\ell}{r_\ell}=\emptyset$.
     \end{enumerate}
 \item We say that $A$ is {\em classical} if $\sum_{n=1}^\infty r_n<\infty,$ $r_0>0$ and for all $k\in S_A$ the following holds$:$
    \begin{enumerate}
        \item $\cb{a_k}{r_k}\subseteq\ob{a_0}{r_0};$
        \item whenever $\ell\in S_A$ with $\ell\neq k,$ we have $\cb{a_k}{r_k}\cap\cb{a_\ell}{r_\ell}=\emptyset$.
    \end{enumerate}
 \item We say $A$ is {\em annular} if $a_0=a_1$ and $r_0>r_1>0$.
 \item The (classical) {\em error set} $E(A)$ of $A$ is defined to be the set
 \[
 \bigcup\limits_{\substack{m,n\in S_A\\m\neq n}}\bigg(\cb{a_m}{r_m}\cap \cb{a_n}{r_n}\bigg)\cup\bigcup\limits_{n\in S_A}((\C\setminus\ob{a_0}{r_0})\cap \cb{a_n}{r_n}).
 \]
\end{enumerate}
For $\alpha\geq 1$ we define the {\em discrepancy function of order $\alpha,$} $\delta_\alpha:\scs\to[-\infty,\infty),$ by
 \begin{equation}\label{gendiscr}\delta_\alpha(A):=r_0^\alpha -\sum\limits_{n=1}^\infty r_n^\alpha\qquad(A=((a_n,r_n))_{n\geq 0}\in\scs).
 \end{equation}
The {\em annular discrepancy function} $\delta_{\rm ann}:\scs\to[-\infty,\infty)$ is defined by
\[
\delta_{\rm ann}(A):=r_0-r_1-2\sum\limits_{n=2}^\infty r_n\qquad(A=((a_n,r_n))_{n\geq 0}\in\scs).
\]
\end{definition}

We usually denote an abstract Swiss cheese by $A=((a_n,r_n))$, where it is understood in context that the sequence is indexed by the non-negative integers. We shall often say {\em annular Swiss cheese} to mean an annular, abstract Swiss cheese. We say a Swiss cheese set $X$ is {\em semiclassical} ({\em classical}) if there is a semiclassical (classical) abstract Swiss cheese $A$ such that $X=X_A$. We note that the annular discrepancy function is defined for all abstract Swiss cheeses, but is only really considered when dealing with annular Swiss cheeses.

As in \cite{FMY}, we introduce the following functions on $\mathcal F$.

\begin{definition}
The {\em radius sum function} is the map $\rho:\scs\to[0,\infty]$ given by
\[
\rho(A):=\sum\limits_{n=1}^\infty r_n\qquad(A=((a_n,r_n))\in\scs).
\]
The {\em centre bound function} is the map $\mu:\scs\to[0,\infty]$ given by
\[
\mu(A):=\sup_{n\in\N}{\abs{a_n}}\qquad(A=((a_n,r_n))\in\scs).
\]
The {\em annular radius sum function} is the map $\ar:\scs\to[0,\infty]$ given by
\[
\ar(A):=\sum\limits_{n=2}^\infty r_n\qquad(A=((a_n,r_n))\in\scs).
\]
\end{definition}

We often impose conditions on the centres and radii. It is easy to see that, for an abstract Swiss cheese $A$, we have $\delta_1(A)>-\infty$ if and only if $\rho(A)<\infty$. Thus the condition $\delta_1(A)>-\infty$ in the definition of a classical Swiss cheese from \cite{feinheath2010} is equivalent to the condition $\rho(A)<\infty$ in Definition \ref{absSCdef}.

We denote the collection of all abstract Swiss cheeses $A=((a_n,r_n))$ with $\rho(A)<\infty$ and $(r_n)_{n=1}^\infty$ non-increasing by $\mathcal N$. In addition, for each $M>0$, we denote the set of all those abstract Swiss cheeses $A\in\mathcal N$ such that $\rho(A)\leq M$ by $\mathcal N_M$.

It is often useful to consider those abstract Swiss cheeses where there are no obvious ``redundant'' open disks. For example, open disks contained in, or equal to, any other open disk. To this end, we make the following definition.

\begin{definition}
Let $A=((a_n,r_n))$ be an abstract Swiss cheese. Then $A$ is {\em redundancy-free}, if for all $k\in S_A$  we have $\ob{a_k}{r_k}\cap\cb{a_0}{r_0}\neq\emptyset$, and for all $\ell\in S_A$ with $k\neq \ell$ we have $\ob{a_k}{r_k}\not\subseteq\ob{a_\ell}{r_\ell}$.
\end{definition}

Let $U\subseteq \mathbb{C}$ be an open set and let $A=((a_n,r_n))$ be an abstract Swiss cheese. We denote by $H_A(U)$ the set of all $n\in S_A$ for which $\cb{a_n}{r_n}\cap U\neq\emptyset$, and we denote by $\rho_U(A)$ the sum $\sum_{n\in H_A(U)}r_n$.

The following result is essentially \cite[Lemma~4.3]{FMY}.

\begin{lemma}\label{nonredundantcheese}
Let $A=((a_n,r_n))\in\mathcal F$ with $\rho(A)<\infty$. Then there exists a redundancy-free abstract Swiss cheese $B=((b_n,s_n))\in\mathcal N_{\rho(A)}$ with $X_B=X_A$ and $\cb{b_0}{s_0}=\cb{a_0}{r_0}$ such that$,$ for all open sets $U\subseteq\C,$ $\rho_U(B)\leq \rho_U(A)$.
\end{lemma}

In the above lemma, the conditions $\cb{b_0}{s_0}=\cb{a_0}{r_0}$ and $\rho(B)\leq\rho(A)$ together imply that $\delta_1(B)\geq\delta_1(A)$. An application of Fatou's lemma shows that, for each $\alpha\geq1$, the function $\delta_\alpha$ is upper semicontinuous on $\mathcal N$. Furthermore, for each $M>0$ and $\alpha>1$, the dominated convergence theorem can be used to show that $\delta_\alpha$ is continuous on $\mathcal N_M$.

Semiclassical (and hence classical) Swiss cheese sets have desirable topological properties such as being rectifiably path connected and locally rectifiably path connected. In fact, in \cite{dalefein2010}, it was proved that for any two points $z,w$ in a classical Swiss cheese set $X$, there is a rectifiable path $\gamma$ in $X$ joining $z$ and $w$ with length at most $\pi\abs{z-w}$; thus, $X$ is {\em uniformly regular} (as defined in \cite{dalefein2010}). In fact, it is easy to see that the same proof works for semiclassical Swiss cheese sets and the constant $\pi$ can be improved to $\pi/2$. These properties imply, in particular, that no semiclassical Swiss cheese set has any isolated points. Note that, for an arbitrary abstract Swiss cheese $A\in\mathcal N$, $X_A$ may have isolated points. Also, as noted in \cite{feinheath2010}, as a consequence of a theorem of \cite{whyburn1958} all classical Swiss cheese sets with empty interior are homeomorphic to the Sierpi\'nski carpet.

It was proved in \cite{feinheath2010} that, for any semiclassical Swiss cheese set $X$, the algebra $R(X)$ is essential. We comment that there are slight differences in definitions of Swiss cheese/abstract Swiss cheese used in \cite{feinheath2010,mason2010} and here. However, all three papers use the same definition for classical/semiclassical Swiss cheese sets. We refer the reader to Section~\ref{Comparison} for a comparison of various notions of a Swiss cheese object.

Due to an argument of Allard (outlined in \cite[p.~163-164]{browder1969}), any classical Swiss cheese set has positive $2$-dimensional Lebesgue measure. With minor adjustments this argument shows that a semiclassical Swiss cheese must also have positive $2$-dimensional Lebesgue measure (area). An alternative argument uses a theorem of Wesler \cite{wesler1960}.

\section{Examples of Swiss cheese sets}
\label{swisscheeseexamples}
We now look at a selection of examples of Swiss cheese sets and their properties from the literature. The Swiss cheese sets discussed here may or may not be classical.
The Swiss cheese sets which we describe as classical were introduced by Alice Roth in \cite{roth1938} (see also \cite{AliceRoth2005}) as examples of compact plane sets $X$ where $R(X)\neq A(X)=C(X)$.
Later these examples were discovered independently by Mergelyan \cite{mergelyan1952}. Since then there have been numerous constructions of compact plane sets, fitting our definition of a Swiss cheese set, designed so that $R(X)$ has various properties.

Prior to the work in \cite{feinheath2010} the focus was on the Swiss cheese set, rather than the underlying (abstract) Swiss cheese. We shall often use the language of abstract Swiss cheeses to describe these examples. The following is a selection of notable Swiss cheese constructions from the literature.

\subsection{Bounded point derivations}
The following construction of a classical Swiss cheese set $X$ such that $R(X)$ admits no non-zero bounded point derivations is due to Wermer \cite{wermer1967}. His construction relies on the following observation. Let $X$ be a compact plane set and let $x\in X$, there exists a non-zero bounded point derivation on $R(X)$ at $x$ if and only if there exists a constant $M>0$ such that $\abs{f'(x)}\leq M\abs f_X$ for each $f\in R_0(X)$. We now briefly outline some key steps of the construction from \cite{wermer1967}.

Let $X_0$ be a finite, classical Swiss cheese set (where only finitely many open disks have been deleted) and let $M,\varepsilon>0$ with $\varepsilon<1/2$. For each $n\in\N$ and for each $j,k\in\Z$, let $D_{jk}^n$ denote the open disk centred at $(j+ki)/n$ of radius $r_n=\varepsilon/n^2$. For a fixed $n\in\N$, we let $S$ denote the set of all $(j,k)\in\Z^2$ such that $\overline{D_{jk}^n}$ lies in the interior of $X_0$. The key to the construction in \cite{wermer1967} is the following lemma.

\begin{lemma}\label{wermerlemma}
There exists an integer $N=N(\varepsilon,M,X_0)$ such that
\[
\sum\limits_{(j,k)\in S}\frac{\varepsilon/N^2}{\abs{z-(j+ki)/N}^2}>M
\]
and
\[
\sum\limits_{(j,k)\in S}\frac{\varepsilon/N^2}{\abs{z-(j+ki)/N}}\leq 50
\]
for all $z\in X':=X\setminus\bigcup_{(j,k)\in S}D_{jk}^N$.
\end{lemma}

Wermer's construction of a classical Swiss cheese set $X$ such that $R(X)$ admits no non-zero bounded point derivations is as follows. Choose a suitable sequence $(\varepsilon_n)$ and let $X_0$ be the closed unit disk. Construct each finite, classical Swiss cheese set $X_{n}$ ($n\geq 1$) inductively by applying Lemma \ref{wermerlemma} to $X=X_{n-1},$ $\varepsilon=\varepsilon_n$ and $M=n$ to and setting $X_n=X'$. Let $X$ be the intersection the sets $X_n$, which is a classical Swiss cheese set. The choice of sequence $(\varepsilon_n)$ allows us to make the sum of the radii of all open disks arbitrarily small. For each $z\in X$, the construction yields a sequence $(f_n)$ of rational functions on $X$ such that $\abs{f_n'(z)}>n$ (from Lemma \ref{wermerlemma}) and $\abs{f}_X\leq 50$ for each $n\in\N$. It then follows that $R(X)$ admits no non-zero bounded point derivations.

This construction was adapted by O'Farrell \cite{o1973isolated}, to construct a Swiss cheese set $X$ for which $R(X)$ admits a non-zero bounded point derivation at exactly one single point of $X$. From the proofs in \cite{wermer1967} and \cite{o1973isolated}, we can distill the following proposition.

\begin{proposition}\label{Anulus_class_no_bpd}
Let $\varepsilon >0$, and $s_0>s_1\geq 0$. Then there exists a classical abstract Swiss cheese $A = ((a_n,r_n))$ with $r_j = s_j$ for $j=1,2$ and $\ar(A)<\varepsilon$ such that $R(X_A)$ has no non-zero bounded point derivations.
\end{proposition}

In fact, Wermer's construction can be applied to any finite, classical Swiss cheese set $X_0$. Proposition \ref{Anulus_class_no_bpd} is obtained by choosing $X_0$ to be either a closed disk or a closed annulus.

The following result, which is due to Browder, is given in \cite{wermer1967}.

\begin{proposition}
Let $A=((a_n,r_n))$ be a classical abstract Swiss cheese and suppose that
$\sum_{n=1}^\infty r_n\log(1/r_n)<\infty.$
Then there exist non-zero bounded point derivations on $X_A$ almost everywhere with respect to Lebesgue measure.
\end{proposition}

In particular, there exists classical Swiss cheese sets $X$ such that $R(X)$ admits non-zero bounded point derivations at almost every point, and classical Swiss cheese sets $Y$ such that $R(Y)$ admits no non-zero bounded point derivations.

Note that a theorem of Hallstrom \cite{hallstrom1969} gives necessary and sufficient conditions for the existence of non-zero bounded point derivations on $R(X)$ at a point $x\in X$, where $X$ is a compact plane set, in terms of analytic capacity. Browder's result is based on a more elementary sufficient condition.

\subsection{Regular uniform algebras}
McKissick \cite{mckissick1963nontrivial} constructed a Swiss cheese set $X$ such that $R(X)$ is regular and non-trivial. In fact, this is the first known example of a non-trivial, regular uniform algebra. His original construction, which relies on a construction of Beurling outlined in \cite[p.~349-355]{stout1971}, was greatly simplified by K\"orner \cite{korner1986cheaper}, which is the version presented below adjusted by translation and scaling.

\begin{lemma}\label{kornerlemma}
Let $\varepsilon>0,$ let $\Delta=\ob ar$ be an open disk and let $E=\C\setminus\Delta$. There exists a sequence of complex numbers $(a_n),$ a sequence of positive real numbers $(r_n)$ and a sequence of rational functions $(f_n)$ such that$,$ if we let $U=\bigcup_{k=1}^\infty \ob{a_k}{r_k},$ then the following hold$:$
\begin{enumerate}
  \item $\sum_{n=1}^\infty r_n<\varepsilon;$
  \item the poles of $f_n$ lie in $U$ for each $n;$
  \item the sequence $(f_n)$ converges uniformly on $\C\setminus U$ to a function $f$ such that $f\left(E\setminus U\right) = \{0\}$ and $0\notin f(\Delta\setminus U);$
  \item $U\subseteq\{z\in\C:r-\varepsilon<\abs{z-a}<r\};$
  \item for each $m,n\in\N$ with $m\neq n$ we have $\ob{a_m}{r_m}\cap\ob{a_n}{r_n}=\emptyset$.
\end{enumerate}
\end{lemma}

Let $\overline{\Delta}$ be a closed disk in $\mathbb C$ with positive radius, and let $(D_m)_{m=1}^\infty$ be an enumeration of open disks in $\C$ with rational centre (a point with rational real and imaginary parts) and rational (non-zero) radii. Then we may apply Lemma~\ref{kornerlemma} to each $D_m$ with $\varepsilon_m =1/2^m$ to obtain $(\Delta_m^{(k)})_{k=1}^\infty$. In this way we obtain a Swiss cheese set $X=\overline{\Delta} \setminus \bigcup_{m,k} \Delta_m^{(k)}$, where $R(X)\neq C(X)$ and $R(X)$ is regular. To see that $R(X)$ is regular, let $K\subseteq X$ be closed and let $z\in X\setminus K$, then there exists $m\in\N$ with $z\in D_m\cap X$ and $D_m\cap K = \emptyset$. From the construction there exists $f\in R(X)$ with $f(z)\neq 0$ and $f$ is identically zero on $K$.

McKissick \cite{mckissick1963nontrivial} used the collection of all open disks with rational centres and radii from the plane to give a Swiss cheese set $X$ such that $R(X)$ is regular. In fact, choosing every open disk with rational centre and rational radii is more than is required. In \cite{o1979regular}, O'Farrell uses those open disks whose centre $a$ has rational real and imaginary parts and rational radius $r$ which satisfy either $0<\abs a\leq 1$, $r<\abs{a}/2$ and $r<1-\abs{a}$, or $a=0$ and $r=2^{-n}$ for some $n\in\N$. Applying Lemma~\ref{kornerlemma} to this collection of open disks to obtain the Swiss cheese set $X$ (starting from the closed unit disk) ensures that $R(X)$ is regular and $0\in X$. With careful control over the sum of the radii of the deleted open disks, O'Farrell showed that $R(X)$ admits a bounded point derivation of infinite order at $0$. Using the method from \cite{mckissick1963nontrivial}, along with Proposition \ref{subsetprops}, (see also \cite{o1979regular}) the following can be proved.

\begin{proposition}\label{mckissickconstruction}
Let $b_0=b_1\in\C$ and $s_0>s_1\geq 0$ and let $\varepsilon>0$. There exists an abstract Swiss cheese $A=((a_n,r_n))$ with $a_j=b_j,$ $r_j=s_j$ for $j=0,1$ and  $\ar(A)<\varepsilon$ such that $R(X_A)$ is regular.
\end{proposition}

Note that Swiss cheese sets obtained by this method are not classical in general. For instance, in McKissick's construction, any deleted open disk contains a sequence of (redundant) deleted open disks.

K\"orner's lemma has been used or adapted to construct a number of examples of Swiss cheese sets $X$ for which $R(X)$ has various combinations of properties. For example, in \cite{feinstein1991}, the first author modified the lemma to construct a Swiss cheese set $X$, obtained using a construction similar to that of O'Farrell, such that $R(X)$ has a prime ideal whose closure is not prime.

In \cite{wang1975}, Wang uses McKissick's lemma to give an example of a Swiss cheese set $X$ such that $R(X)$ is strongly regular at a non-peak point.  In \cite{feinstein2001}, the first author gave an example of a Swiss cheese set $X$ such that $R(X)$ has no non-trivial Jensen measures yet is not regular. In \cite{feinstein2004}, the same author gave an example of a Swiss cheese set $X$, using Wermer's construction (Proposition \ref{Anulus_class_no_bpd}), such that $R(X)$ has no non-zero bounded point derivations but is not weakly amenable.
This construction was improved by Heath \cite{Heath2005} who gave an example of a compact plane set $Y$ such that $R(Y)$ was regular and admitted no non-zero bounded point derivations but is not weakly amenable. However, in this example disks were deleted from a square shaped compact set in the plane rather than a closed disk.

\subsection{Other examples}
There are other examples based on the construction of suitable Swiss cheese sets. In \cite{steen1966}, Steen gave a construction which can be used to give a classical Swiss cheese set $X$ such that $R(X)$ is not antisymmetric.

There are several examples of classical Swiss cheese sets (with our current definition) given in \cite{gamelin1984}. These include the {\em roadrunner set} (p.~52), the {\em string of beads} (p.~146), the {\em stitched disk} (Example 9.3), the {\em Champagne bubbles} (p.~227), along with Examples 9.1 and 9.2. Example 9.2 of \cite{gamelin1984} was also used in \cite{dalefein2010}. These examples, unlike most of the above examples, have non-empty interior.

Swiss cheese sets, and similar constructions, have also been used to construct interesting Banach algebras of functions, for example in \cite{brennan1973} and \cite{Brennan2013}, and in harmonic approximation, for example in \cite{browder1969}. The latter example is a ``square Swiss cheese'', obtained by deleting from a closed, square shaped, compact plane set $K$ a sequence of open, square shaped, plane sets.

\section{Classicalisation theorems}
We have seen that semiclassical and classical Swiss cheeses have desirable topological properties. In \cite{feinheath2010}, Feinstein and Heath gave a sufficient condition to find a classical Swiss cheese whose associated plane set is a subset of the original. In fact they proved the following, stated here in the language of abstract Swiss cheeses.

\begin{proposition}[Classicalisation theorem]\label{classicalisationtheorem}
Let $A$ be an abstract Swiss cheese with $\delta_1(A)>0$. Then there exists a classical abstract Swiss cheese $B$ such that $\delta_1(B)\geq\delta_1(A)$ and $X_B\subseteq X_A$.
\end{proposition}

This theorem was later proved using a transfinite induction by Mason \cite{mason2010}. In \cite{FMY}, the current authors proved the classicalisation theorem by considering a compact set $\mathcal S$ in the topological space $\scs$. The classical abstract Swiss cheese is obtained by first maximising $\delta_1$ on $\mathcal S$ and then minimising $\delta_2$ on the resulting compact subset of $\mathcal S$. These functions are upper-semicontinuous and continuous, respectively on $\mathcal S$. The compact set $\mathcal S$ depends on the initial abstract Swiss cheese, and consists of those abstract Swiss cheeses which are ``good candidates'' for the final, classical abstract Swiss cheese; see \cite{FMY} for formal definition of $\mathcal S$. In the same paper a similar result (below) for annular Swiss cheeses was proved by this topological method. This result can also be proved using transfinite induction.

\begin{proposition}[Annular classicalisation theorem]\label{annularclassicalisation}
Let $A=((a_n,r_n))$ be an annular Swiss cheese with $\delta_{\rm ann}(A)>0$. Then there exists a classical$,$ annular Swiss cheese $B=((b_n,s_n))$ with $b_0=a_0$ such that $\delta_{\rm ann}(B)\geq\delta_{\rm ann}(A)$ and $X_B\subseteq X_A$.
\end{proposition}

For a non-empty compact set $K$ and $M>0,$ we define
\[
U(K,M):=\{z\in\C:\dist(z,K)<M\}.
\]
Let $I$ be a non-empty subset of $\mathbb{N}$, and let $(K_n)_{n\in I}$ be a countable collection of non-empty, compact sets and $(M_n)_{n\in I}$ a countable collection of positive real numbers; we write $U_n$ for $U(K_n,M_n)$ in what follows. We say the countable collection of pairs $((K_n,M_n))_{n\in I}$ is {\em admissible} (with respect to $A$) if $\rho_{U_n}(A)<M_n/2$ for all $n\in I,$ ${U_m}\cap {U_n}=\emptyset$ for all $m\in I$ with $m\neq n$ and $\overline{U_n}\subseteq\ob{a_0}{r_0}$ for all $n\in I$.  By constructing a suitable compact subset of $\scs$ the following was proved. Note that this result can also be proved using transfinite induction.

\begin{proposition}[Controlled classicalisation theorem]\label{localclassicalisation}
Let $A=((a_n,r_n))$ be a redundancy-free abstract Swiss cheese with $\rho(A)<\infty$ and let $I$ be a non-empty subset of $\mathbb N$. Suppose that $((K_n,M_n))_{n\in I}$ is an admissible collection of pairs with $E(A)\subseteq\bigcup_{n\in I} K_n$. Then there exists a classical abstract Swiss cheese $B=((b_n,s_n))$ with $\delta_1(B)\geq\delta_1(A),$ $X_B\subseteq X_A,$ and the following hold$:$
\begin{enumerate}
    \item for all $k\in S_A\setminus\bigcup_{n\in I}H_A(U_n)$ there exists $\ell \in S_B$ with $\ob{b_\ell}{s_\ell} = \ob{a_k}{r_k};$
    \item $\rho_{U_n}(B)\leq \rho_{U_n}(A)$ for all $n\in I$.
\end{enumerate}
\end{proposition}

By combining Propositions \ref{annularclassicalisation} and \ref{localclassicalisation} we obtain, through a sequence of approximations described in \cite{FMY}, the following improvement of Proposition \ref{mckissickconstruction}.

\begin{proposition}\label{classicalmckissick}
Let $b_0=b_1\in\C$ and $s_0>s_1\geq 0$ and let $\varepsilon>0$. There exists a classical abstract Swiss cheese $A=((a_n,r_n))$ with $a_j = b_j,$ $r_j=s_j$ for $j=1,2$ and $\sum_{n=2}^\infty r_n<\varepsilon$ such that $R(X_A)$ is regular.
\end{proposition}

\section{Comparison of Swiss cheeses}
\label{Comparison}
At present there are two distinct abstract notions of a Swiss cheese; a Swiss cheese in the sense of \cite{feinheath2010,mason2010} and the abstract Swiss cheese as in Definition~\ref{absSCdef}. Recall that a Swiss cheese (as in \cite{feinheath2010,mason2010}) is a pair consisting of a closed disk and a countable collection of open disks. In this definition, all disks are assumed to be non-degenerate (have positive radius). We can describe any Swiss cheese set using a Swiss cheese or an abstract Swiss cheese. There is also a related notion of a disk assignment map, introduced in \cite{mason2010}, which can also be used to describe Swiss cheese sets. We now explain the relationship between these different notions.

In \cite{FMY}, it was described how, given an abstract Swiss cheese $A=((a_n,r_n))$, we can obtain an associated Swiss cheese $\mathbf D_A$ by setting
\[
\mathbf D_A=(\cb{a_0}{r_0},\{\ob{a_n}{r_n}:n\in S_A\}).
\]
In this way we can obtain any Swiss cheese. The mapping of the collection of all abstract Swiss cheeses onto the collection of all Swiss cheeses, described above, is a surjection (necessarily many-to-one) and preserves the associated Swiss cheese set. We also have $\delta(\mathbf D_A)\geq \delta_1(A)$, where $\delta(\mathbf D_A)$ is defined by
\begin{equation}\label{FHDiscrepancy}
\delta(\mathbf D):=r(\overline{\Delta})-\sum_{D\in\mathcal D}r(D)\qquad(\mathbf D=(\overline{\Delta},\mathcal D))
\end{equation}
and $r(D)$ denotes the radius of the disk $D$. (The quantity $\delta(\mathbf D)$ in \eqref{FHDiscrepancy} is called the {\em discrepancy} of $\mathbf D$.) However, we cannot obtain every abstract Swiss cheese from a Swiss cheese. For instance, in a Swiss cheese, the collection of open disks may not contain any repetitions, whereas an abstract Swiss cheese can contain repeated pairs. Moreover, an abstract Swiss cheese can contain degenerate pairs (where $r_n=0$), which is not allowed in the definition of Swiss cheeses.

In \cite{mason2010}, Mason considered {\em disk assignment functions} $d:S\to\mathcal O$, where $S\subseteq\No$ with  $0\in S$, $\mathcal O$ denotes the collection of all open disks and complements of closed disks in the plane, and $\mathbf E_d:=\{\mathbb C \setminus d(0),d(S\setminus \{0\})\}$ is a Swiss cheese.  Note that disk assignment functions allow for repeated disks, whereas the Feinstein-Heath Swiss cheese does not. All disks considered in \cite{mason2010} were assumed to have positive radius. Such a function has the {\em Feinstein-Heath condition} if the Swiss cheese $\mathbf E_d$ (as shown above) satisfies $\delta(\mathbf E_d)>0$, with $\delta(\mathbf E_d)$ given by \eqref{FHDiscrepancy}. There is a surjection from the set of all abstract Swiss cheeses $A=((a_n,r_n))$ with $r_0>0$ onto the set of all disk assignment functions. To construct this map, take $S=S_A\cup\{0\}$ and use the obvious map. In particular, this map preserves the radius sum and the associated Swiss cheese set.

Each Swiss cheese set which can be obtained from an abstract Swiss cheese $A$ with $\delta_1(A)>0$ can also be obtained from a Swiss cheese $\mathbf D=(\overline{\Delta},\mathcal D)$ with $\delta(\mathbf D)>0$ (as defined in \eqref{FHDiscrepancy}). In addition, we can also obtain such a Swiss cheese set from a disk assignment function $d$ satisfying the Feinstein-Heath property.

There are notable difference between the three methods used to prove the Feinstein-Heath classicalisation theorem, which is stated as Proposition~\ref{classicalisationtheorem}. In \cite{feinheath2010}, Feinstein and Heath constructed a partial order on a suitable collection and used Zorn's lemma to obtain a maximal object with respect to this partial order, which was associated with a classical Swiss cheese. Mason \cite{mason2010} used transfinite induction to construct a family of self maps on the set $H$ of all disk assignment functions with the Feinstein-Heath condition, which eventually stabilised at a disk assignment function associated with a classical Swiss cheese. In \cite{FMY}, a compact collection of abstract Swiss cheeses was constructed so that maximising the discrepancy function $\delta_1$ and then minimising the function $\delta_2$ yields a classical abstract Swiss cheese.

The classical Swiss cheese obtained by Mason's method need not be maximal in the sense of the Feinstein-Heath partial order, maximise the function $\delta_1$ or minimise the function $\delta_2$. We refer the reader to the respective papers for more details on these approaches. Consider the following example.

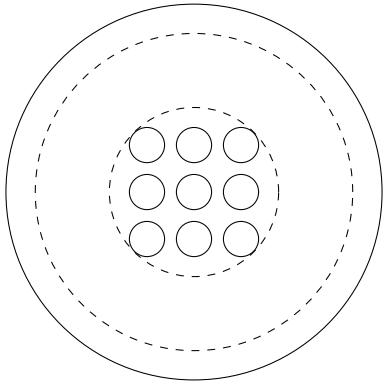
\begin{figure}[htbp]
\begin{tikzpicture}
\draw(   0,   0) circle [radius= 2.5];
\draw (-0.625,-0.625) circle [radius=0.234];
\draw (   0,-0.625) circle [radius=0.234];
\draw (0.625,-0.625) circle [radius=0.234];
\draw (-0.625,   0) circle [radius=0.234];
\draw (   0,   0) circle [radius=0.234];
\draw (0.625,   0) circle [radius=0.234];
\draw (-0.625,0.625) circle [radius=0.234];
\draw (   0,0.625) circle [radius=0.234];
\draw (0.625,0.625) circle [radius=0.234];
\draw[dashed] (   0,   0) circle [radius=1.125];
\draw[dashed] (   0,   0) circle [radius=2.11];
\end{tikzpicture}
\caption{A classical Swiss cheese set where the discrepancy can be improved.}
\label{counterexample_to_FH_Mason_equiv}
\end{figure}

Let $\overline{\Delta}$ be the closed unit disk, let $\varepsilon>0$ be small. Let $z_1,\dotsc,z_9$ be the distinct points whose real and imaginary parts are either $0$ or $\pm1/4$. Take $r=3/32$. Let $D_k$ be the open disks centred at $z_k$ of radius $r$ for $k=1,2,\dotsc,9$, as shown in Figure \ref{counterexample_to_FH_Mason_equiv}. Then the $\overline{D_k}$ are disjoint, the resulting Swiss cheese (or abstract Swiss cheese) is classical and satisfies the conditions of the Feinstein-Heath classicalisation theorem. However, applying Mason's construction will yield the same Swiss cheese, but both the Feinstein-Heath approach and the abstract Swiss cheese approach will yield new Swiss cheeses. However, these new (abstract) Swiss cheeses can be different.

From the definition of the partial order in \cite{feinheath2010}, it is easy to see that if an abstract Swiss cheese $A$ maximises $\delta_1$ and minimises $\delta_2$, on a suitable compact collection of Swiss cheeses (see \cite{FMY}), then the corresponding Swiss cheese must be maximal in the Feinstein-Heath partial order.

The example above, illustrated by Figure \ref{counterexample_to_FH_Mason_equiv}, can be used to show that a Swiss cheese which is maximal in the Feinstein-Heath partial order need not maximise discrepancy. The dashed lines (in Figure \ref{counterexample_to_FH_Mason_equiv}) show the maximum and minimum disks which could be used to replace the collection of smaller disks. If, for example, we form two Swiss cheeses by replacing the collection by the maximum and minimum disks, then these Swiss cheeses are not comparable in the Feinstein-Heath partial order. The (abstract) Swiss cheese obtained by replacing by the minimal disk maximises discrepancy, but replacing by the maximal disk does not change the discrepancy.

\section{Semiclassicalisation}
Let $A$ be an abstract Swiss cheese with positive discrepancy. By the Feinstein-Heath theorem, we can find a classical abstract Swiss cheese $B$ such that $X_B\subseteq X_A$; in particular, we can obtain a semiclassical abstract Swiss cheese. We describe an inductive process where at each step we seek to increase the discrepancy of the abstract Swiss cheese by combining overlapping open disks and/or pulling in the closed disk. We show that this process produces a sequence which converges to a semiclassical abstract Swiss cheese. This process was originally described in the third author's MSc dissertation at the University of Nottingham.

The following elementary lemmas are very minor modification of the lemmas in \cite{FMY}, and we omit details of the proofs, and are illustrated in Figure~\ref{elementary_lemma_fig}.

\begin{lemma}\label{combinediscs}
Let $a_1,a_2\in \mathbb C$ and $r_1$ and $r_2$ be positive real numbers such that $\ob{a_1}{r_1}\cap \ob{a_2}{r_2}\neq \emptyset$. Then there exists a unique pair $(a,r)\in \mathbb{C}\times (0,\infty)$ such that $r< r_1+r_2,$ $\ob{a_1}{r_1}\cup \ob{a_2}{r_2}\subseteq \ob{a}{r},$ and $r$ is minimal.
\end{lemma}

\begin{lemma} \label{outsidedisc}
Let $a_1,a_2\in \mathbb C$ and $r_1,r_2>0$ be such that $\ob{a_1}{r_1}\nsubseteq \ob{a_2}{r_2}$ and $\ob{a_2}{r_2}\nsubseteq \ob{a_1}{r_1}$. Then there exists a unique pair $(a,r)\in \mathbb{C}\times (0,\infty)$ such that $r>r_1-r_2$,  $\cb{a}{r}\subseteq \cb{a_1}{r_1}$,  $\cb{a}{r}\cap \ob{a_2}{r_2}=\emptyset$, and $r$ is maximal.
\end{lemma}

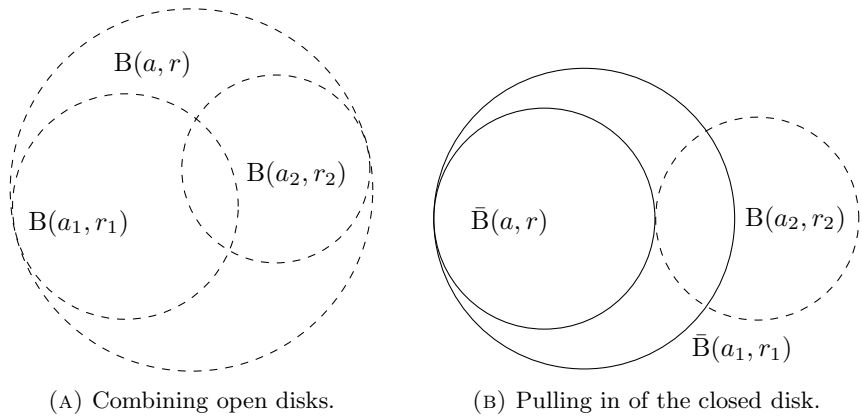
\begin{figure}[htbp]
\centering
\begin{subfigure}[b]{0.45\textwidth}\centering
\begin{tikzpicture}
\draw[dashed] (-0.5,   0) circle [radius= 1.5];
\draw[dashed] ( 1.5, 0.5) circle [radius=1.25];
\draw[dashed] (0.379,0.22) circle [radius=2.41];
\node at (-1.12,-0.219) {$\ob{a_1}{r_1}$};
\node at (1.78,0.43) {$\ob{a_2}{r_2}$};
\node at (-0.136,1.89) {$\ob{a}{r}$};
\end{tikzpicture}
\caption{Combining open disks.}
\label{combininglem}
\end{subfigure}
~
\begin{subfigure}[b]{0.45\textwidth}
	\centering
\begin{tikzpicture}
\draw(-0.5,   0) circle [radius=   2];
\draw[dashed] ( 1.8,   0) circle [radius=1.35];
\draw (-1.03,   0) circle [radius=1.47];
\node at (-1.49,-0.0409) {$\cb{a}{r}$};
\node at (2.31,-0.00585) {$\ob{a_2}{r_2}$};
\node at (1.6,-1.73) {$\cb{a_1}{r_1}$};
\end{tikzpicture}
\caption{Pulling in of the closed disk.}
\label{pullinginlem}
\end{subfigure}
\caption{Elementary lemmas for combining and pulling in disks.}
\label{elementary_lemma_fig}
\end{figure}

We call the open disk $\ob{a}{r}$ in Lemma \ref{combinediscs} the \emph{minimal radius open disk} containing $\ob{a_1}{r_1}\cup\ob{a_2}{r_2}$, and we call the closed disk $\cb{a}{r}$ in Lemma \ref{outsidedisc} the \emph{maximal radius closed subdisk} of $\cb{a_1}{r_1}\setminus\ob{a_2}{r_2}$.

Let $D_1$ and $D_2$ be open disks in $\C$ with $D_1\cap D_2 \neq \emptyset$. If we draw a line through the centres of these two disks (if the centres coincide, we just draw any line through the centre), we define the length of the line segment in $D_1\cap D_2$ to be the \emph{length of overlap}. See Figure~3 for an illustration. Let $A=((a_n,r_n))$ be an abstract Swiss cheese with finite discrepancy, and assume for some $k,\ell\in S_A$ we have $\ob{a_k}{r_k}\cap \ob{a_\ell}{r_\ell}\neq \emptyset$. Then by replacing these two open disks by the minimal radius open disk containing them, we can increase the discrepancy of $A$ by half  the length of overlap of these two open disks. A more formal formulation of replacing (and also discarding) disks is to be made in the proof of Theorem~\ref{Semiclassicalisation}.

If there exist pairs of open disks in $A$ which have non-empty intersection, there exists a pair of them with maximal length of overlap. To see this, first notice that the supremum length of overlap is finite because $\delta_1(A)$ is finite, and we denote it by $M$. If $M=0$ then there is nothing to prove. Otherwise, there are only finitely many pairs of open disks with length of overlap larger than $M/3$, because $\delta_1(A)$ is finite. Then among these pairs of open disks, we can find one pair that achieves maximal length of overlap. Let $B_1$, $B_2$, $D_1$ and $D_2$ be open disks in $\C$ with $B_1\subseteq D_1$, $B_2\subseteq D_2$ and $B_1\cap B_2 \neq \emptyset$. We claim that the length of overlap of $B_1$ and $B_2$ is no larger than the length of overlap of $D_1$ and $D_2$. The proof is elementary and we leave the details to the reader.

Let $\overline{D_0}$ be a closed disk and $D_1$ be an open disk in $\C$, both with positive radii, such that $D_0\nsubseteq D_1$ and $D_1\nsubseteq D_0$. If we draw a line through the centres of these two disks, we define the length of the line segment contained in $D_1\setminus \overline{D_0}$ to be the \emph{extrusion length} of $D_1$ from $\overline{D_0}$. See Figure~3 for an illustration. Let $A=((a_n,r_n))$ be an abstract Swiss cheese with positive discrepancy, and assume for some $k\in S_A$ we have $\ob{a_k}{r_k}\nsubseteq \cb{a_0}{r_0}$. Then by replacing $\cb{a_0}{r_0}$ by the maximal radius closed subdisk of $\cb{a_0}{r_0}\setminus\ob{a_k}{r_k}$ and discarding $\ob{a_k}{r_k}$, we can increase the discrepancy of $A$ by half the extrusion length of $\ob{a_k}{r_k}$ from $\cb{a_0}{r_0}$. If there exists $k\in\N$ with $\ob{a_k}{r_k}\nsubseteq\cb{a_0}{r_0}$ then there exists $\ell\in\N$ such that the length of extrusion of $\ob{a_\ell}{r_\ell}$ from $\cb{a_0}{r_0}$ is maximal. (The proof is similar to that for the maximal length of overlap above.)  Let $\overline{B_0}$, $\overline{D_0}$ be closed disks and $B_1$ and $D_1$ be open disks in $\C$, with $\overline{B_0}\subseteq \overline{D_0}$, $D_1\subseteq B_1$ and $D_1\nsubseteq \overline{D_0}$. Then the extrusion length of $D_1$ from $\overline{D_0}$ is no larger than the extrusion length of  $B_1$ from $\overline{B_0}$.

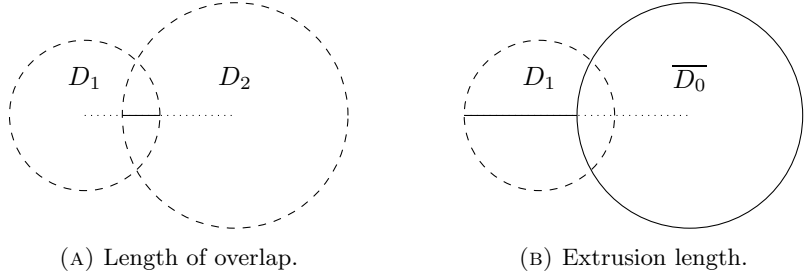
\begin{figure}[htbp] \label{length}
\centering
\begin{subfigure}[b]{0.45\textwidth}\centering
\begin{tikzpicture}
\draw[dashed] (0,   0) circle [radius= 1];
\draw[dashed] ( 2, 0) circle [radius=1.5];
\draw[dotted] (0,0) -- (2, 0) ;
\draw (0.5,0)--(1,0);
\node at (0,0.5) {$D_1$};
\node at (2,0.5) {$D_2$};
\end{tikzpicture}
\caption{Length of overlap.}
\label{length_of_overlapping}
\end{subfigure}
~
\begin{subfigure}[b]{0.45\textwidth}
	\centering
\begin{tikzpicture}
\draw[dashed] (0,   0) circle [radius= 1];
\draw ( 2, 0) circle [radius=1.5];
\draw[dotted] (-1,0) -- (2, 0) ;
\draw (-1,0)--(0.5,0);
\node at (0,0.5) {$D_1$};
\node at (2,0.5) {$\overline{D_0}$};
\end{tikzpicture}
\caption{Extrusion length.}
\label{length_of_exterior}
\end{subfigure}
\caption{Length of overlap and extrusion length shown by the unbroken line.}
\end{figure}

\begin{theorem}[Semiclassicalisation] \label{Semiclassicalisation}
Let $A=((a_\ell,r_\ell))$ be an abstract Swiss cheese with $\delta_1(A)>0$. Then there exists a semiclassical abstract Swiss cheese $B$ with $X_B\subseteq X_A$ and $\delta_1(B)\geq \delta_1(A)$.
\end{theorem}

We remark that Theorem~\ref{Semiclassicalisation} is a corollary of the Feinstein-Heath classicalisation theorem (Proposition~\ref{classicalisationtheorem}). However, the proof given here uses an inductive construction that terminates at the first countable ordinal. In contrast, the proof given in \cite{mason2010} used a transfinite induction which terminates before the first uncountable ordinal.

\begin{proof}
We construct a suitable sequence of abstract Swiss cheeses $(A^{(m)})_{m=1}^\infty$ by induction such that $\delta_1(A^{(m)})$ is non-decreasing and $X_{A^{(m+1)}}\subseteq X_{A^{(m)}}$. In this construction, we apply, where appropriate, Lemmas \ref{combinediscs} and \ref{outsidedisc} alternately.

Let $A^{(0)}=A$, and assume we have constructed $A^{(2m)}$ for some $m\geq 0$. If, for all distinct $k,\ell\in S_{A^{(2m)}}$, we have $\ob{a^{(2m)}_k}{r^{(2m)}_k}\cap \ob{a^{(2m)}_\ell}{r^{(2m)}_\ell}=\emptyset$, then let $A^{(2m+1)}=A^{(2m)}$. Otherwise we can find $k,\ell\in S_{A^{(2m)}}$, with $k<\ell$ such that $\ob{a^{(2m)}_k}{r^{(2m)}_k}\cap \ob{a^{(2m)}_\ell}{r^{(2m)}_\ell}\neq \emptyset$ and they achieve maximal length of overlap. In this case we set $A^{(2m+1)}=((a^{(2m+1)}_n,r^{(2k+1)}_n))$ where $a^{(2m+1)}_n=a^{(2m)}_n$ and $r^{(2m+1)}_n = r^{(2m)}_n$ for $n \neq k,\ell$; $a^{(2m+1)}_k, r^{(2m+1)}_k$ equal to the centre and radius of the minimal radius open disk containing $\ob{a^{(2m)}_k}{r^{(2m)}_k}\cup \ob{a^{(2m)}_\ell}{r^{(2m)}_\ell}$; and $a^{(2m+1)}_\ell=0$, $r^{(2m+1)}_\ell=0$. (Note that the technique of allocating the minimal radius open disk to the smaller index is due to Mason \cite{mason2010}.)

If, for all $k\in S_{A^{(2m+1)}}$, we have $\ob{a^{(2m+1)}_k}{r^{(2m+1)}_k}\subseteq \cb{a^{(2m+1)}_0}{r^{(2m+1)}_0}$, we set $A^{(2m+2)}=A^{(2m+1)}$. Otherwise, we can find $k\in S_{A^{(2m+1)}}$ such that
\[
\ob{a^{(2m+1)}_k}{r^{(2m+1)}_k}\nsubseteq \cb{a^{(2m+1)}_0}{r^{(2m+1)}_0}
\]
and the extrusion length of $\ob{a^{(2m+1)}_k}{r^{(2m+1)}_k}$ from $\cb{a^{(2m+1)}_0}{r^{(2m+1)}_0}$ is maximal. In this case let $A^{(2m+2)} = ((a^{(2m+2)}_n, r^{(2m+2)}_n))$ where $a^{(2m+2)}_n = a^{(2m+1)}_n$,  $r^{(2m+2)}_n = r^{(2m+1)}_n$ for $n \neq 0,k$; $a^{(2m+2)}_0, r^{(2m+2)}_0$ equal to the centre and radius of the maximal radius closed subdisk of $\cb{a^{(2m+1)}_0}{r^{(2m+1)}_0}\setminus \ob{a^{(2m+1)}_k}{r^{(2m+1)}_k}$; $a^{(2m+2)}_k=0$ and $r^{(2m+2)}_k=0$.

In this way we have constructed a sequence of abstract Swiss cheeses $(A^{(m)})_{m\geq 1}$. Let $m\in\No$ and $k\in\N$. If $\ob{a_k^{(2m+1)}}{r_k^{(2m+1)}}$ is not equal to $\ob{a_k^{(2m)}}{r_k^{(2m)}}$ then either $\ob{a_k^{(2m)}}{r_k^{(2m)}}\subseteq\ob{a_k^{(2m+1)}}{r_k^{(2m+1)}}$ or
\begin{equation}\label{masontrick}
\ob{a_k^{(2m)}}{r_k^{(2m)}}\subseteq\ob{a_\ell^{(2m+1)}}{r_\ell^{(2m+1)}}\quad\text{for some $\ell\in\N$ with $\ell<k$}
\end{equation}
and $r_k^{(m')}=0$ for all $m'\geq 2m+1$. If $\ob{a_k^{(2m+1)}}{r_k^{(2m+1)}}\neq\ob{a_k^{(2m+2)}}{r_k^{(2m+2)}}$ then we must have $\ob{a_k^{(2m+1)}}{r_k^{(2m+1)}}\subseteq\C \setminus\cb{a_0^{(2m+2)}}{r_0^{(2m+2)}}$ and $r_k^{(m')}=0$ for all $m'\geq 2m+2$.

We notice that the sequence of closed disks $(\cb{a^{(m)}_0}{r^{(m)}_0})_m$ is nested decreasing, thus we have $a_0^{(m)}\to b_0\in\C$ and $r_0^{(m)}\to s_0\geq 0$. Such limits exist according to \cite[Proposition~2.3]{feinheath2010}.
For each $k\geq 1$, we observe that the sequence of open disks $(\ob{a^{(m)}_k}{r^{(m)}_k})_{m\geq 1}$ is either bounded and nested increasing, or there exists $N\geq 1$ such that, for all $m\geq N$, we have $a^{(m)}_k=r^{(m)}_k=0$. In both cases, we see the sequences $(a^{(m)}_k)_{m}$ and $(r^{(m)}_k)_{m}$ converge and we denote the limits by $b_k$ and $s_k$, respectively. Letting $B = ((b_n,s_n))$ we have constructed the limit abstract Swiss cheese. We see that $A^{(m)}\to B$ in $\mathcal{F}$ with the product topology.

From the construction it is clear that $\delta_1(A^{(m)})\geq \delta_1(A)$ and $(\delta_1({A^{(m)}}))$ is non-decreasing. Since $\delta_1$ is upper semicontinuous and $A^{(m)}\to B$ we see that $\delta_1(B)\geq \delta_1(A)>0$. In particular, we have $s_0>0$. We show that $X_B\subseteq X_A$. It is clear that $\cb{b_0}{s_0}\subseteq \cb{a_0}{r_0}$. Let $z\in\C\setminus X_{A}$, we show $z\notin X_B$. If $z\notin\cb{b_0}{s_0}$ then $z\in\C\setminus X_B$. Otherwise, we can find $k\in S_{A}$ such that $z\in\ob{a_k}{r_k}\cap\cb{b_0}{s_0}$. Suppose there exists $N\in \mathbb N$ such that $r_k^{(m)}=0$ for all $m\geq N$. Since $\cb{a_0^{(m)}}{r_0^{(m)}}$ has non-empty intersection with $\ob{a_k}{r_k}$, by \eqref{masontrick}, there exists a non-increasing sequence $(\ell_m)_{m\geq N}$ of integers with $1\leq\ell_m<k$ such that $\ob{a_k}{r_k}\subseteq\ob{a_{\ell_m}^{(m)}}{r_{\ell_m}^{(m)}}$ for all $m\geq N$. This sequence $(\ell_m)_{m\geq N}$ is eventually constant, say $\ell_m=\ell$ for all $m\geq N_1$. Thus $\ob{a_k}{r_k}\subseteq\ob{a_\ell^{(m)}}{r_\ell^{(m)}}$ for all $m\geq N_1$ and it follows that $\ob{a_k}{r_k}\subseteq\ob{b_\ell}{s_\ell}$. Hence $z\notin X_B$ as required.

We show that $B$ is a semiclassical abstract Swiss cheese.
Assume towards a contradiction that $B$ is not semiclassical. There are two cases. The first case is there exists $k,\ell\in S_B$ such that $\ob{b_k}{s_k}\cap \ob{b_\ell}{s_\ell} \neq \emptyset$. Then there exists $N\geq 1$ such that for all $m\geq N$ we have $\ob{a^{(m)}_k}{r^{(m)}_k}\cap \ob{a^{(m)}_\ell}{r^{(m)}_\ell} \neq \emptyset$. Let $L_m$ be the length of overlap of the disks $\ob{a^{(m)}_k}{r^{(m)}_k}$ and $\ob{a^{(m)}_\ell}{r^{(m)}_\ell}$,
which is positive and non-decreasing. From the inductive construction, it is clear that $\delta_1(A^{(2m+1)})-\delta_1(A^{(2m)})\geq L_{2m}/2$ for all $m\geq N$, which contradicts $\delta_1(A^{(m)})\leq r_0$ for all $m$. The second case is there exists $k\in S_B$ such that $\ob{b_k}{s_k}$ is not contained in $\cb{b_0}{s_0}$. Then there exists $N\geq 1$ such that $\ob{a^{(m)}_k}{r^{(m)}_k}$ is not contained in $\cb{a^{(m)}_0}{r^{(m)}_0}$ for all $m\geq N$. Let $I_m$ be the extrusion length of $\ob{a^{(m)}_k}{r^{(m)}_k}$ from $\cb{a^{(m)}_0}{r^{(m)}_0}$. Clearly $I_m>0$ and is non-decreasing. For all $m\geq N$ we have $\delta_1(A^{(2m+2)})-\delta_1(A^{(2m+1)})\geq I_{2m+1}/2$, which is a contradiction since $\delta_1(A^{(M)})\leq r_0$ for all $m$.
\end{proof}

\section{A classical counterexample to the conjecture of S. E. Morris}
In \cite{feinstein2004}, Feinstein gave a counterexample to a conjecture of S. E. Morris by constructing a Swiss cheese set $X$ where $R(X)$ has no non-zero, bounded point derivations but $R(X)$ is not weakly amenable. What he proved is the following.

\begin{theorem} \label{Joel_Morris_Cheese}
For each $C>0$ there is a compact plane set $X$ obtained by deleting from the closed unit disk a countable union of open disks such that the unit circle $\mathbb{T}$ is a subset of $X,$ $R(X)$ has no non-zero$,$ bounded point derivations$,$ but for all $f,g$ in $R_0(X),$
\begin{equation} \label{cont_deri_dual}
\left\lvert\int_\mathbb{T} f'(z)g(z)dz\right\rvert \leq C \abs{f}_X \abs{g}_X.
\end{equation}
\end{theorem}

The existence of a non-zero bounded derivation from $R(X)$ to its dual space is an easy consequence of the estimate in (\ref{cont_deri_dual}) (see also \cite[p. 2390]{feinstein2004}). The theorem follows from the following lemma, which is also proved in \cite{feinstein2004}.

\begin{lemma}\label{Joel_Cauchy}
Let $D_n$ be a sequence of open disks in $\mathbb{C}$ (not necessarily pairwise disjoint) whose closures are contained in the open unit disk. Set $X=\overline{\Delta}\backslash \bigcup_{n=1}^\infty D_n$, where $\overline{\Delta}$ is the closed unit disk. Let $d_n$ be the distance from $D_n$ to $\mathbb{T}$ and let $r_n$ be the radius of $D_n$. Let $f$ and $g$ be in $R_0(X)$. Then
\[
\left\lvert\int_\mathbb{T} f'(z)g(z)dz\right\rvert\leq 4 \pi \abs{f}_X \abs{g}_X \sum_{n=1}^\infty \frac{r_n}{d_n^2}.
\]
\end{lemma}

In this section we prove a new version of Theorem \ref{Joel_Morris_Cheese} where the Swiss cheese set $X$ is classical. This follows our general classicalisation scheme as discussed in \cite{FMY}. The main theorem of this section is the following.

\begin{theorem} \label{SEMorris}
For each $C>0,$ there exists a classical abstract Swiss cheese $B=((b_n,s_n))$ such that $R(X_B)$ has no non-zero bounded point derivations, $s_0=1,$ the unit circle $\mathbb{T}\subseteq X_B,$  and $\sum_{n=1}^\infty s_n/d_n^2 \leq C,$ where $d_n$ is the distance from the disk $\ob{b_n}{s_n}$ to $\mathbb{T}$ if $s_n>0$ and $d_n=1$ if $s_n=0$.
\end{theorem}
\begin{proof}
We construct a classical abstract Swiss cheese $B=((b_n,s_n))$ such that $R(X_B)$ has no nonzero bounded point derivations and
\[\sum_{n=1}^\infty \frac{s_n}{d_n^2}<C.\]

For each $n\in \mathbb{N}$, let $A^{(n)}=((a_m^{(n)},r_m^{(n)}))$ be a classical annular Swiss cheese with $r_0^{(n)}=(n+1)/(n+2)$, $r_1^{(n)}=n/(n+1)$, $a_0^{(n)}=a_1^{(n)}=0$,
\[ \ar(A^{(n)})<\min \left\{ \frac C {2^{n+3}(n+3)^2},  \frac 1{12(n+3)^2} \right\}, \]
and such that $R(X_{A^{(n)}})$ has no non-zero bounded point derivations.

For each $n\geq 1$, set
\[ K_n = \left\{ z\in \mathbb{C}: \abs{z}\in \left[ \frac {n+1}{n+2}-\frac 1{4(n+3)^2}, \frac{n+1}{n+2}+\frac{1}{4(n+3)^2}\right]\right\}.\]
We also choose, for each $n\geq 1$, a sequence of open disks such that the annular Swiss cheese
$ B^{(n)}=((b_m^{(n)},s_m^{(n)}))$ with $b_0^{(n)}=b_1^{(n)}=0$, \[ s_0^{(n)}=\frac{n+1}{n+2}+\frac 1{4(n+3)^2},\] \[ s_1^{(n)}=\frac{n+1}{n+2}-\frac 1{4(n+3)^2},\]
\[ \ar(B^{(n)})<\min \left\{ \frac C{2^{n+3}(n+3)^2}, \frac 1{12(n+3)^2}\right\},\]
is classical, and such that $R(X_{B^{(n)}})$ has no non-zero bounded point derivations.

We construct an abstract Swiss cheese $A=((a_m,r_m))$ such that $a_0=0$, $r_0=1$, $(r_m)$ is an enumeration of $(r_m^{(n)})_{n\geq 1, m\geq 2}$ and $(s_m^{(n)})_{n\geq 1, m\geq 2}$, and $(a_m)$ is the enumeration of $(a_m^{(n)})_{n\geq 1,m\geq 2}$ and $(b_m^{(n)})_{n\geq 1,m\geq 2}$ corresponding to $(r_m)$. By Lemma~\ref{nonredundantcheese}, there exists a redundancy-free abstract Swiss cheese $A'=((a'_m,r'_m)) \in \mathcal N$ such that $X_{A'} = X_A$ and $\rho_U(A')\leq \rho_U(A)$ for all open subset $U$ of $\mathbb C$. Notice that for fixed $n$, both disks $\ob{a_m^{(n)}}{r_m^{(n)}}$ and $\ob{b_m^{(n)}}{s_m^{(n)}}$ are contained in the disk $\ob{0}{(n+2)/(n+3)}$.
We have
\[ \sum_{m=1}^\infty \frac {r'_m}{(d'_m)^2} \leq \sum_{n=1}^\infty \left( (n+3)^2 \sum_{m=2}^\infty (r_m^{(n)} + s_m^{(n)})\right) \leq \frac C 4,\] where $d'_m$ is the distance from the disk $\ob{a'_m}{r'_m}$ to $\mathbb T$ if $r'_m>0$ and $d'_m=1$ otherwise.

Set $M_n = 1/(4(n+3)^2)$, and let $U_n=\{z\in \mathbb{C} : \dist(z,K_n)< M_n\}.$ We observe that
\begin{align} \label{8_1} \rho_{U_n}(A')\leq \rho_{U_n}(A) &\leq \ar(A^{(n)})+\ar(A^{(n+1)})+\ar(B^{(n)}) \\ \nonumber &<\min\left\{\frac 1{4(n+3)^2}, \frac{3C}{2^{n+3}(n+3)^2}\right\}. \nonumber
\end{align}
Then $((K_n,M_n))_{n\in \mathbb N}$ is an admissible collection of pairs for the abstract Swiss cheese $A'$, which satisfies the conditions in Proposition~\ref{localclassicalisation}. See Figure \ref{semorrisconjdiagram} for an illustration of a resulting pair $(K_n,U_n)$. Thus, by Proposition \ref{localclassicalisation}, there exists a classical abstract Swiss cheese $B=((b_n,s_n))$ such that $\delta_1(B)\geq\delta_1(A')$, $X_B\subseteq X_{A'}$, $b_0=0$, $s_0=1$ and $\rho_{U_n}(B)\leq \rho_{U_n}(A')$ for all $n\in \mathbb N$.

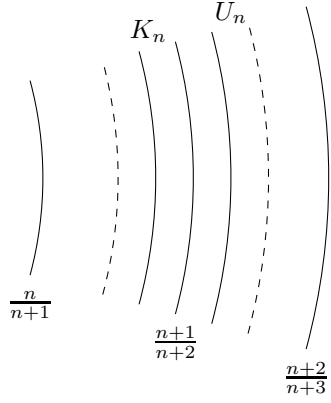
\begin{figure}[htbp]
\centering
\begin{tikzpicture}
\draw (4.8296,-1.2941) arc [radius=5, start angle=-15, end angle= 15];
\node at (4.8296,-1.7) {$\frac{n}{n+1}$};
\draw (6.7615,-1.8117) arc [radius=7, start angle=-15, end angle= 15];
\node at (6.7615,-2.2) {$\frac{n+1}{n+2}$};
\draw (8.5001,-2.2776) arc [radius=8.8, start angle=-15, end angle= 15];
\node at (8.5001,-2.67) {$\frac{n+2}{n+3}$};
\draw (7.2444,-1.9411) arc [radius=7.5, start angle=-15, end angle= 15]; 
\draw (6.2785,-1.6823) arc [radius=6.5, start angle=-15, end angle= 15];
\node at (6.4,1.95) {$K_n$};
\draw[dashed] (5.7956,-1.5529) arc [radius=6, start angle=-15, end angle= 15];
\draw[dashed] (7.7274,-2.0706) arc [radius=8, start angle=-15, end angle= 15];
\node at (7.5,2.2) {$U_n$};
\end{tikzpicture}
\caption{A pair $(K_n,U_n)$ as in the proof of Theorem~\ref{SEMorris}.}
\label{semorrisconjdiagram}
\end{figure}

We have
\[
\sum_{n=1}^\infty \frac{s_n}{d_n^2} = \sum_{n\in S_1} \frac{s_n}{d_n^2} + \sum_{n\in S_2} \frac{s_n}{d_n^2},
\]
where $S_2 := \bigcup_{n=1}^\infty H_B( U_n)$ and $S_1= \mathbb{N}\backslash S_2$; note here that $H_B(U_m)\cap H_B(U_n) = \emptyset$ for all $m\neq n$ . For all $n\in S_1$ we have $\ob{b_n}{s_n}=\ob{a'_m}{r'_m}$ for some $m\geq 1$. Then we have
\begin{equation} \label{summation_1} \sum_{n\in S_1} \frac{s_n}{d_n^2} \leq \sum_{n=1}^\infty \frac{r'_n}{(d'_n)^2}\leq \frac C4.\end{equation} On the other hand, from the construction we have
\[ \sum_{n \in S_2} \frac{s_n}{d_n^2} = \sum_{n=1}^\infty \left( \sum_{m\in S_2^{(n)}} \frac{s_m}{d_m^2}\right),\]
where $S_2^{(n)} := H_A(U_n).$
For each $m\in S_2^{(n)}$, since \[s_m\leq \rho_{U_n}(B)\leq \rho_{U_n}(A')<1/(4(n+3)^2)\] by \eqref{8_1} we have $\cb{b_m}{s_m}\subseteq \ob{0}{(n+2)/(n+3)}$, so $d_m>1/(n+3)$. Again by \eqref{8_1} we observe that
\[ \sum_{m\in S_2^{(n)}} s_m = \rho_{U_n}(B)\leq \rho_{U_n}(A')<\frac {3C}{2^{n+3}(n+3)^2}.\]
Therefore we have
\[ \sum_{m\in S_2} \frac{s_m}{d_m^2} = \sum_{n=1}^\infty \left(\sum_{m\in S_2^{(n)}} \frac{s_m}{d_m^2}\right) \leq \sum_{n=1}^\infty \sum_{m\in S_2^{(n)}}(n+3)^2s_m <\frac C2.\]
Combining with \eqref{summation_1} we conclude that
\[ \sum_{n=1}^\infty \frac{s_n}{d_n^2}<C.\]
This concludes the proof.
\end{proof}

\section{Open questions}
We raise the following open questions. Let $X$ be a compact plane set.

\begin{question}
Let $B$ be the classical abstract Swiss cheese constructed in Theorem~$\ref{SEMorris}$. Can $R(X_B)$ be regular? Must $R(X_B)$ be regular?
\end{question}

\begin{question}
If $R(X)$ has no non-zero bounded point derivations, must $R(X)$ be regular$?$
\end{question}

\begin{question}
If $R(X)$ is weakly amenable, must $R(X)$ be trivial? The same question is open for uniform algebras.
\end{question}


\providecommand{\bysame}{\leavevmode\hbox to3em{\hrulefill}\thinspace}
\providecommand{\MR}{\relax\ifhmode\unskip\space\fi MR }
\providecommand{\MRhref}[2]{%
  \href{http://www.ams.org/mathscinet-getitem?mr=#1}{#2}
}
\providecommand{\href}[2]{#2}

\end{document}